\documentclass{article}
\usepackage[affil-it]{authblk}
\usepackage{amsthm, graphicx, latexsym}  
\usepackage[utopia]{mathdesign}      %% latex says amssymb and amsfonts shouldn't be used with mathdesign/mdput (\usepackage{...})
 \usepackage[mathscr]{eucal}
\usepackage{amsmath, yhmath}  
\usepackage{array, longtable, sidecap, wrapfig, subfig, color}
\usepackage[normalem]{ulem}
\usepackage{cancel}
\usepackage{tikz}

\usepackage[colorlinks=true, pdfstartview=FitV, linkcolor=blue, citecolor=blue, urlcolor=cerulean]{hyperref}
\DeclareMathOperator{\closure}{closure}
\DeclareMathOperator{\Area}{Area}

\usepackage{fancyhdr}
\begin{document}

\newtheorem{thm}{Theorem}
\newtheorem*{thm*}{Theorem}
\newtheorem{cor}[thm]{Corollary}
\newtheorem{lemma}[thm]{Lemma}
\newtheorem{lemmadef}[thm]{Lemma/Definition}
\newtheorem{proposition}[thm]{Proposition}
\newtheorem{prop}[thm]{Proposition}

\newenvironment{definition}[1][Definition]{\begin{trivlist}
\item[\hskip \labelsep {\bfseries #1}]}{\end{trivlist}}
\newenvironment{example}[1][Example]{\begin{trivlist}
\item[\hskip \labelsep {\bfseries #1}]}{\end{trivlist}}
\newenvironment{remark}[1][Remark]{\begin{trivlist}
\item[\hskip \labelsep {\bfseries #1}]}{\end{trivlist}}
\newenvironment{remarks}[1][Remarks]{\begin{trivlist}
\item[\hskip \labelsep {\bfseries #1}]}{\end{trivlist}}

\theoremstyle{definition}
\newtheorem{defn}[thm]{Definition}
\newtheorem{defns}[thm]{Definitions}
\newtheorem{rem}[thm]{Remark}
\newtheorem{Rem}[thm]{Remark}
\newtheorem{Rems}[thm]{Remarks}
\newtheorem{Exercise}[thm]{Exercise}
\newtheorem{Example}[thm]{Example}
\newtheorem{Examples}[thm]{Examples}
\newtheorem{Open questions}[thm]{Open questions}
\newtheorem{Open question}[thm]{Open question}
\newtheorem{Open problems}[thm]{Open problems}
\newtheorem{Open problem}[thm]{Open problem}
\newtheorem*{Acknowledgements}{Acknowledgements}  

\newcommand{\leftexp}[2]{{\vphantom{#2}}^{#1}\!{#2}}
\newcommand{\beq}{\begin{eqnarray*}}
\newcommand{\eeq}{\end{eqnarray*}}
\newcommand{\ol}{\overline}
\newcommand{\Supp}{\textrm{Supp}}
\newcommand{\Span}{\textrm{Span}}
\newcommand{\Gr}{\textbf{Gr} }
\newcommand{\GL}{\textrm{GL}}
\newcommand{\SL}{\textrm{SL}}
\newcommand{\Matr}{\textbf{Matr}}
\newcommand{\mat}[9]{\left(\begin{array}{ccc}#1&#2&#3\\#4&#5&#6\\#7&#8&#9\end{array}\right)}
\newcommand{\rowm}[3]{\left(\begin{array}{c}#1\\#2\\#3\end{array}\right)}
\newcommand{\rk}{\textrm{rk}~}
\newcommand{\im}{\textrm{im}~}
\newcommand{\Div}{\textrm{div}~}
\newcommand{\z}{\overline}
\newcommand{\Par}[2]{\frac{\partial#1}{\partial#2}}
\newcommand{\hf}{\frac{1}{2}}
\newcommand{\lf}{\left(}
\newcommand{\rt}{\right)}
\newcommand{\re}{\textrm{Re}}
\newcommand{\img}{\textrm{Im}}
\newcommand{\limn}{{\lim\atop{n\to\infty}}}
\newcommand{\Res}{\textrm{Res}}
\newcommand{\eps}{\epsilon}
\newcommand{\tp}{\frac{1}{2\pi i}}
\newcommand{\of}{\circ}
\newcommand{\ub}{\underbrace}
\newcommand{\K}{\tilde{k}}
\newcommand{\floor}[1]{{\lfloor#1\rfloor}}
\newcommand{\Gal}{\textrm{Gal}}
\newcommand{\Aut}{\textrm{Aut}}
\newcommand{\out}[1]{{\textrm{Out}(F_#1)}}
\newcommand{\pr}[2]{\langle#1,#2\rangle}
\newcommand{\Tr}{\textrm{Tr}}
\newcommand{\id}{\mathfrak}
\newcommand{\rad}{\mathfrak{R}}
\newcommand{\Spec}{\mathrm{Spec}}
\newcommand{\tensor}{\otimes}
\newcommand{\Ann}{\textrm{Ann}}
\renewcommand{\phi}{\varphi}
\newcommand{\Ker}{\textrm{Ker}~}
\newcommand{\Img}{\textrm{Im}~}
\newcommand{\Z}{\mathbb{Z}}
\newcommand{\N}{\mathbb{N}_{\geq0}}
\newcommand{\trans}[1]{{\leftexp{T}{#1}}}
\def\onto{{\kern3pt\to\kern-8pt\to\kern3pt}}
\newcommand{\IO}{\textrm{IO}}
\newcommand{\IA}{\textrm{IA}}
\newcommand{\Stab}{\textrm{Stab}}
\newcommand{\ssim}[1]{\stackrel{(#1)}{\sim}}
\newcommand{\Lk}[1]{\textrm{Lk}_<(#1)}
\newcommand{\set}[1]{\left\{#1\right\}}
\newcommand{\ssm}{\smallsetminus}
\newcommand{\abs}[1]{\left|#1\right|}
\newcommand{\Dist}{\textup{Dist}}

\newcommand{\wh}{\widehat}
\newcommand{\wt}{\widetilde}
\newcommand{\wb}{\overline}
\newcommand{\ihat}{\hat{\imath}}
\newcommand{\fhat}{\hat{f}}
\newcommand{\ibar}{\overline{\imath}}
\newcommand{\fbar}{\overline{f}}

\def\ms{\medskip}
\renewcommand{\ss}{\smallskip}
\newcommand{\bs}{\bigskip}

\newcommand{\tc}[2]{\textcolor{#1}{#2}}
\definecolor{cerulean}{rgb}{0,.48,.65} \newcommand{\cerulean}[1]{\tc{cerulean}{#1}}
\definecolor{magenta}{rgb}{.5,0,.5} \newcommand{\magenta}[1]{\tc{magenta}{#1}}
\definecolor{dred}{rgb}{.5,0,0} \newcommand{\dred}[1]{\tc{dred}{#1}}
\definecolor{green}{rgb}{0,.5,0} \newcommand{\green}[1]{\tc{green}{#1}}
\definecolor{blue}{rgb}{0,0,0.5} \newcommand{\blue}[1]{\tc{blue}{#1}}
\definecolor{black}{rgb}{0,0,0} \newcommand{\black}[1]{\tc{black}{#1}}
\definecolor{dgreen}{rgb}{0,.3,0} \newcommand{\dgreen}[1]{\tc{dgreen}{#1}}
\definecolor{vdred}{rgb}{.3,0,0} \newcommand{\vdred}[1]{\tc{vdred}{#1}}
\definecolor{red}{rgb}{1,0,0} \newcommand{\red}[1]{\tc{red}{#1}}
\definecolor{salmon}{rgb}{0.98,0.50,0.45} \newcommand{\salmon}[1]{\tc{salmon}{#1}}
\definecolor{gray}{rgb}{.5,.5,.5} \newcommand{\gray}[1]{\tc{gray}{#1}}
\definecolor{seagreen}{rgb}{0.13,0.70,0.67} \newcommand{\seagreen}[1]{\tc{seagreen}{#1}}
\definecolor{chartreuse}{rgb}{0.40,0.80,0.00}\newcommand{\chartreuse}[1]{\textcolor{chartreuse}{#1}}
\definecolor{cornflower}{rgb}{0.39,0.58,0.93} \newcommand{\cornflower}[1]{\textcolor{cornflower}{#1}}
\definecolor{gold}{rgb}{0.80,0.68,0.00}\newcommand{\gold}[1]{\textcolor{gold}{#1}}

\newcommand{\commentO}[1]{\marginpar{\tiny\begin{center}\cerulean{#1}\end{center}}} 
\newcommand{\commentT}[1]{\marginpar{\tiny\begin{center}\gold{#1}\end{center}}}

\setlength{\parindent}{0pt}
\setlength{\parskip}{7pt}

\title{The Generalized Dehn Property does not imply \\ a linear isoperimetric inequality}
\author{Owen Baker and Timothy Riley\thanks{The second author is grateful for support from Simons Collaboration Grant 318301.}}

\date \today

\maketitle

\begin{abstract}
\noindent The Dehn property for a complex is that every non-trivial disk diagram has spurs or shells.  It implies a linear isoperimetric inequality.  It has been conjectured that the same is true of a more general property which also allows cutcells.  We give counterexamples. 

\medskip

\noindent La propri\'et\'e Dehn pour un complexe est que chaque diagramme de disque non trivial a des \'eperons ou des shells. Cela implique une in\'egalit\'e isop\'erim\'etrique lin\'eaire. Il a \'et\'e suppos\'e qu'il en \'etait de m\^eme pour une propri\'et\'e plus g\'en\'erale qui autorise \'egalement les cellules de coupe. Nous pr\'esentons des contre-exemples.

\ms \noindent   \textbf{2010 Mathematics Subject Classification:  20F67, 57M20}  
% 20F65   	Geometric group theory [See also 05C25, 20E08, 57Mxx]
% 20F10   	Word problems, other decision problems, connections with logic and automata
%	20F67   	Hyperbolic groups and nonpositively curved groups
% 57M20          Two-dimensional complexes
\end{abstract}

A \emph{disk diagram} $D\subset S^2$ is a finite contractible 2-complex.
The boundary path $\partial_p D$ is the attaching map of the 2-cell $R_\infty$ giving the decomposition $S^2=D\cup R_\infty$.
A \emph{disk diagram in a 2-complex} $X$ is a combinatorial map $D\to X$.
It is \emph{minimal area} if it contains the minimum number of 2-cells among all $D\to X$ with the same restriction to $\partial_p D$.  Minimal diagrams are \emph{reduced}: there do not exist adjacent `back-to-back' 2-cells in $D$ overlapping in at least an edge $e$, mapping to the same cell in $X$ with the same induced edge sequence on the boundaries starting from $e$. A \emph{spur} is a valence $1$ vertex in $\partial D$.  
A \emph{shell} is a 2-cell $R$ for which more than half the perimeter forms a contiguous segment of the boundary attaching map $\partial_pD$.   An $X$ in which every minimal nontrivial disk diagram contains a shell or spur is said to satisfy the \emph{Dehn property}.

Removing a shell $R$ (more precisely, removing the interior of $R$ from $D$ and the longer portion of the boundary) decreases the area of $D$ by one and the perimeter by at least one.
Removing a 1-cell $S$ leading to a spur   (more precisely,  taking the closure of $D - S$)  decreases perimeter by two. So the Dehn property gives rises to a linear isoperimetric inequality:  that is, there exists $K>0$ (in this case $K=1$) such that every null-homotopic edge-loop in $X$ (equivalently, every edge-loop in the universal cover $\widetilde{X}$) of length at most $n$ admits a disk diagram with at most $Kn$ 2-cells.        

(We do not consider our 2-complexes $X$,  disk diagrams, or the boundary circuits of the 2-cells that comprise them to have base points.)

 Gaster and Wise \cite{gasterwise} declare that
a 2-complex $X$ satisfies the \emph{generalized Dehn property} if each minimal disk diagram $D\to X$ is either a single 0-, 1-, or 2-cell, 
or contains a spur, shell, or \emph{cutcell}---a type of 2-cell which we shall describe momentarily.
They conjecture \cite[Conjecture~2.2]{gasterwise} that if $X$ is a compact 2-complex with this property, then its 
universal cover $\widetilde{X}$ enjoys a linear isoperimetric inequality. The likely intuition behind the conjecture is that a cutcell will be the site of a neck or of branching in a diagram.

Gaster and Wise propose   three  definitions of increasing strength for when a 2-cell  $R$ in a disk diagram $D$ is a \emph{cutcell}:  
\begin{enumerate}
\renewcommand{\labelenumi}{(\arabic{enumi}) }
\item  $D-\closure(R)$ has more than one component.  \label{one}
\item The preimage of $\partial R$ in $\partial_pD$ consists of more than one component. \label{two}
\item (Strong) The preimage of $\partial R$ in $\partial_pD$ consists of more than one component, each a nontrivial path. \label{three}
\end{enumerate}
Wise \cite{wisep} suggests adding the  assumption that each 2-cell $R$ in $X$ embeds in the universal cover  $\widetilde{X}$---that is, $(R,\partial R) \to (X,X^{(1)})$ lifts to an embedding  $(R,\partial R) \to (\widetilde{X},\widetilde{X}^{(1)})$.  This renders definitions \eqref{one} and  \eqref{two}   equivalent.

We  present a counterexample to the strong form:

\begin{thm} \label{main_counterexample}
The presentation 2-complex $X$ associated to the presentation
$$\langle \, a_1, a_2, b_1, b_2, c_1, c_2, c_3 \ | \  a_2b_1b_2a_2^{-1}a_1^{-1}b_2^{-1}c_1c_2c_3,\ \,a_1^{-1}b_1c_1c_2c_3 \, \rangle$$ of  $G = \Z^2 \ast F_4$  
enjoys the strongest form of the generalized Dehn property (definition \eqref{three} of cutcell).  Moreover, both 2-cells in $X$ embed in its universal cover  $\widetilde{X}$.   
\end{thm}

First we will present an example which is more elementary, but where one of the  2-cells of $X$ fails to embed in the universal cover  $\widetilde{X}$.

\begin{thm}  \label{basic_counterexample}
The presentation 2-complex $X$ associated to the presentation  $\langle a, b\,|\,[a,b]c, c \rangle$ of  $G = \Z^2$  enjoys the weakest form of the generalized Dehn property   (definition \eqref{one} of cutcell).   
\end{thm}

Both $\Z^2 \ast F_4$ and $\Z^2$ have quadratic Dehn functions, so the complexes $X$ of Theorems~\ref{main_counterexample} and ~\ref{basic_counterexample} do not admit linear isoperimetric functions. 

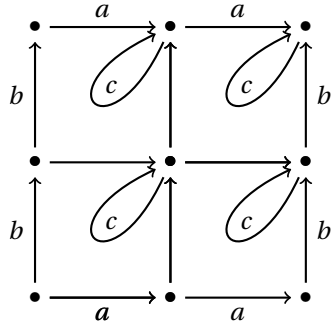
\begin{figure}[h]  \centering
\begin{tikzpicture}[scale=0.6]
  \node (A) at (0,0) {$\bullet$};
	\node (C) at (3,0) {$\bullet$};
	\node (F) at (0,3) {$\bullet$};
	\node (H) at (3,3) {$\bullet$};
	\node (AA) at (3,0) {$\bullet$};
	\node (CC) at (6,0) {$\bullet$};
	\node (FF) at (3,3) {$\bullet$};
	\node (HH) at (6,3) {$\bullet$};
	\node (TA) at (0,3) {$\bullet$};
	\node (TC) at (3,3) {$\bullet$};
	\node (TF) at (0,6) {$\bullet$};
	\node (TH) at (3,6) {$\bullet$};
	\node (TAA) at (3,3) {$\bullet$};
	\node (TCC) at (6,3) {$\bullet$};
	\node (TFF) at (3,6) {$\bullet$};
	\node (THH) at (6,6) {$\bullet$};

  \draw[thick,->] (A) -- (C) node[below,midway] {$a$};
	\draw[thick,->] (F) -- (H);% node[below,midway] {$a$};
	\draw[thick,->] (A) -- (F) node[left,midway] {$b$};
	\draw[thick,->] (C) -- (H);% node[left,midway] {$b$};
	\draw[thick,->] (AA) -- (CC) node[below,midway] {$a$};
	\draw[thick,->] (FF) -- (HH);
	\draw[thick,->] (AA) -- (FF);
	\draw[thick,->] (CC) -- (HH) node[right,midway] {$b$};
	\draw[thick,->] (A) -- (C) node[below,midway] {$a$};
	\draw[thick,->] (TF) -- (TH) node[above,midway] {$a$};
	\draw[thick,->] (TA) -- (TF) node[left,midway] {$b$};
	\draw[thick,->] (TC) -- (TH);% node[left,midway] {$b$};
	\draw[thick,->] (TAA) -- (TCC);% node[below,midway] {$a$};
	\draw[thick,->] (TFF) -- (THH) node[above,midway] {$a$};
	\draw[thick,->] (TAA) -- (TFF);
	\draw[thick,->] (TCC) -- (THH) node[right,midway] {$b$};
	\path (THH) edge [thick,out=245,in=205,looseness=30,->] node [above right] {$c$} (THH);
	\path (TH) edge [thick,out=245,in=205,looseness=30,->] node [above right] {$c$} (TH);
	\path (H) edge [thick,out=245,in=205,looseness=30,->] node [above right] {$c$} (H);
	\path (HH) edge [thick,out=245,in=205,looseness=30,->] node [above right] {$c$} (HH);
\end{tikzpicture} 
\caption{A diagram over  $\langle a, b\,|\,[a,b]c, c \rangle$.} \label{easy figure}
\end{figure}

The example of Theorem~\ref{basic_counterexample}  illuminates the difference between the three definitions  of cutcell.  In the diagram depicted in Figure~\ref{easy figure} each of the four pentagonal cells  are cutcells in the sense of (1) since removing
the closure of any one disconnects a monogon.  But they are not cutcells in the sense of (2) or (3).   For an example with cutcells in the sense of (1) and (2), but not (3), see the remark at the end of this note.   The cutcells discussed in our proof of Theorem~\ref{main_counterexample} are cutcells in all three senses.  

In general, disk-diagrams need not be topological disks.  They can have 1-dimensional portions or cut-vertices.  The following lemma tells us that we may  focus on disk diagrams that are topological disks.  

\begin{lemma} \label{disk lemma}
For the generalized Dehn property (with any of the three definitions of cutcell), it suffices to restrict attention to   minimal disk diagrams $D \to X$ where $D$ is a topological disk. 
\end{lemma}

\begin{proof} If $D$ is 1-dimensional and is not a single vertex, then it possesses spurs.   
Otherwise, consider a maximal subdiagram $D_0$ of $D$ which is a topological disk. A cutcell in $D_0$ will be a cutcell in $D$.  If a shell in $D_0$ is not a shell in $D$, then it is a cutcell in $D$.
\end{proof}

\begin{proof}[Proof of Theorem~\ref{basic_counterexample}]  The universal cover $\widetilde{X}$ is a plane tessellated in the manner of Figure~\ref{easy figure}.    Suppose     $D\to\widetilde{X}$ is the lift of a reduced topological-disk diagram $D \to X$  such that $D$  is not   a single 2-cell.  By Lemma~\ref{disk lemma} it is enough to show that $D$ contains  a cutcell in the sense of definition (1) or contains a shell.  Now $D$ cannot consist  only of monogons, for it   would then have to be a single monogon.  So $D$ includes a pentagonal 2-cell.   One such pentagonal 2-cell $R$ must map to a square in  $\widetilde{X}$ with the property that the square above it and the square to its right do not contain the image of pentagons from $D$.  It follows (since $D$ is reduced and a topological disk) that its upper and right edges are in $\partial D$.   If $R$ encloses a monogon, then $R$ is a cutcell.
If $R$ does not enclose a monogon, a portion of its boundary labelled $a c^{-1} b^{-1}$ is part of the boundary circuit of $D$ and so $R$ is a shell.        
\end{proof}

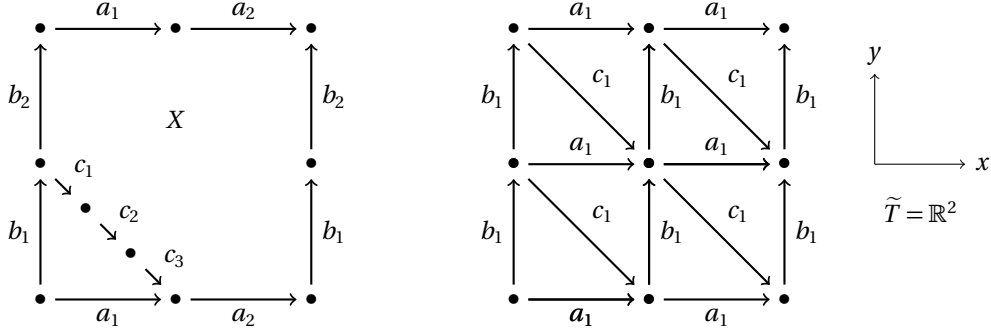
\begin{figure}[h] \centering
\begin{tikzpicture}[scale=0.6]
  \node (A) at (0,0) {$\bullet$};
	\node (B) at (3,0) {$\bullet$};
	\node (C) at (6,0) {$\bullet$};
	\node (D) at (0,3) {$\bullet$};
	\node (E) at (6,3) {$\bullet$};
	\node (F) at (0,6) {$\bullet$};
	\node (G) at (3,6) {$\bullet$};
	\node (H) at (6,6) {$\bullet$};
	\node (I) at (1,2) {$\bullet$};
	\node (J) at (2,1) {$\bullet$};
	\node (K) at (3,4) {$X$};
  \draw[thick,->] (A) -- (B) node[below,midway] {$a_1$};
	\draw[thick,->] (F) -- (G) node[above,midway] {$a_1$};
	\draw[thick,->] (B) -- (C) node[below,midway] {$a_2$};
	\draw[thick,->] (G) -- (H) node[above,midway] {$a_2$};
	\draw[thick,->] (A) -- (D) node[left,midway] {$b_1$};
	\draw[thick,->] (C) -- (E) node[right,midway] {$b_1$};
  \draw[thick,->] (D) -- (F) node[left,midway] {$b_2$};
	\draw[thick,->] (E) -- (H) node[right,midway] {$b_2$};
	\draw[thick,->] (D) -- (I) node[above right,midway] {$c_1$};
	\draw[thick,->] (I) -- (J) node[above right,midway] {$c_2$};
	\draw[thick,->] (J) -- (B) node[above right,midway] {$c_3$};
\end{tikzpicture}
\hspace{40pt}
\begin{tikzpicture}[scale=0.6]
  \node (A) at (0,0) {$\bullet$};
	\node (C) at (3,0) {$\bullet$};
	\node (F) at (0,3) {$\bullet$};
	\node (H) at (3,3) {$\bullet$};
	\node (AA) at (3,0) {$\bullet$};
	\node (CC) at (6,0) {$\bullet$};
	\node (FF) at (3,3) {$\bullet$};
	\node (HH) at (6,3) {$\bullet$};
	\node (TA) at (0,3) {$\bullet$};
	\node (TC) at (3,3) {$\bullet$};
	\node (TF) at (0,6) {$\bullet$};
	\node (TH) at (3,6) {$\bullet$};
	\node (TAA) at (3,3) {$\bullet$};
	\node (TCC) at (6,3) {$\bullet$};
	\node (TFF) at (3,6) {$\bullet$};
	\node (THH) at (6,6) {$\bullet$};
	\node (L) at (9,2) {$\widetilde{T}=\mathbb{R}^2$};

  \draw[thick,->] (A) -- (C) node[below,midway] {$a_1$};
	\draw[thick,->] (F) -- (H) node[above,midway] {$a_1$};
	\draw[thick,->] (A) -- (F) node[left,midway] {$b_1$};
	\draw[thick,->] (C) -- (H) node[right,midway] {$b_1$};
	\draw[thick,->] (F) -- (C) node[above right,midway] {$c_1$};
	\draw[thick,->] (AA) -- (CC) node[below,midway] {$a_1$};
	\draw[thick,->] (FF) -- (HH);
	\draw[thick,->] (AA) -- (FF);
	\draw[thick,->] (CC) -- (HH) node[right,midway] {$b_1$};
	\draw[thick,->] (FF) -- (CC) node[above right,midway] {$c_1$};
	\draw[thick,->] (A) -- (C) node[below,midway] {$a_1$};
	\draw[thick,->] (TF) -- (TH) node[above,midway] {$a_1$};
	\draw[thick,->] (TA) -- (TF) node[left,midway] {$b_1$};
	\draw[thick,->] (TC) -- (TH) node[right,midway] {$b_1$};
	\draw[thick,->] (TF) -- (TC) node[above right,midway] {$c_1$};
	\draw[thick,->] (TAA) -- (TCC) node[above,midway] {$a_1$};
	\draw[thick,->] (TFF) -- (THH) node[above,midway] {$a_1$};
	\draw[thick,->] (TAA) -- (TFF);
	\draw[thick,->] (TCC) -- (THH) node[right,midway] {$b_1$};
	\draw[thick,->] (TFF) -- (TCC) node[above right,midway] {$c_1$};
	
	\draw[->] (8,3) -- (8,5) node[above] {$y$};
	\draw[->] (8,3) -- (10,3) node[right] {$x$};
\end{tikzpicture}
\caption{The presentation $2$-complex $X$  of Theorem~\ref{main_counterexample} (left) and a portion of a tessellation of the plane which is a quotient of the universal cover $\widetilde{X}$ (right).} \label{harder figure}
\end{figure}  

\begin{proof}[Proof of Theorem~\ref{main_counterexample}]
The presentation 2-complex $X$ is depicted leftmost in Figure~\ref{harder figure}. It has two 2-cells, a ``triangle'' of perimeter 5 and a ``pentagon'' of perimeter 9.
If we let $a=a_1a_2$ and $b=b_1b_2$ and eliminate $a_1=aa_2^{-1}$,
$b_1=bb_2^{-1}$, and $c_1=b_1^{-1}a_1c_3^{-1}c_2^{-1}$, we obtain
$$G \ \cong \ \langle a,b,a_2,b_2,c_2,c_3\,|\,[a,b]\rangle \ \cong \ \mathbb{Z}^2\ast F_4.$$

That both 2-cells in $X$ embed in its universal cover  $\widetilde{X}$ is then evident from the observation that every subword one reads along a proper portion of its boundary circuit represents a non-trivial element in $\mathbb{Z}^2\ast F_4$.

Let $\pi:G\to\mathbb{Z}^2$ be the projection   killing the factor $F_4$ and sending $a$ and $b$ to $(1,0)$ and $(0,1)$, respectively.
Thus $\pi(a_1)=(1,0)$, $\pi(b_1)=(0,1)$, $\pi(c_1)=(1,-1)$ and $\pi(a_2)=\pi(b_2)=\pi(c_2)=\pi(c_3)=(0,0)$.
This map $\pi$ is induced by    the combinatorial map $\rho$ from $X$ to the  presentation complex $T$ (a torus) of $\langle a_1,b_1,c_1\,|b_1a_1^{-1}c_1,a_1^{-1}b_1c_1\rangle$ obtained by collapsing $a_2,b_2,c_2,c_3$.
The cell complex structure on $\widetilde{T}=\mathbb{R}^2$ is that of the unit square grid with each square divided into two triangles, as shown on the right in Figure~\ref{harder figure}. The lift $\widetilde{\rho}:\widetilde{X}\to\widetilde{T}$ sends ``pentagons'' to the upper-right triangles in $\widetilde{T}$
and adjacent (via the path labelled $c_1c_2c_3$) ``triangles'' to lower-left triangles of the same square.

Now consider a reduced  topological-disk diagram $D\to\widetilde{X}$ which is not   a single 2-cell.  By Lemma~\ref{disk lemma},  to establish the  generalized Dehn property, it suffices to show that $D$ contains a shell or cutcell. 

Let $p=(x,y)\in\widetilde{T}$ be the point of the image under $\widetilde{\rho}$ with minimal $x$ among all points with minimal $y$.
Let $R$ be a 2-cell of the disk diagram whose image in $\widetilde{T}$ contains $p$.
We consider cases.
If $R$ is a pentagon then its left, bottom and diagonal faces (i.e.\ edges) are free.
     Otherwise the image in $\widetilde{T}$ of an adjacent 2-cell would contradict our choice of $p$ or the fact that $D$ is minimal area and so reduced.
     Therefore $R$ is a shell in this case.
If, instead, $R$ is a triangle then, similarly, its left and bottom faces are free.
     So (as $D$ is reduced) it abuts a pentagon along its length-3 diagonal boundary segment.
     But then this pentagon is a (strong) cutcell.
So $X$ satisfies the (strong) generalized Dehn property.
\end{proof}

We conclude with three remarks.   
 
 \begin{Rem} The presentation
\begin{equation}
\langle \, a_1,  b_1,  c_1, c_2, c_3 \ | \   b_1 a_1^{-1} c_1c_2c_3,\ \,a_1^{-1}b_1c_1c_2c_3 \, \rangle \ \cong \  \Z^2 \ast F_2,
\label{third presentation}
\end{equation} 
obtained by killing $a_2$ and  $b_2$ in $G$ of Theorem~\ref{main_counterexample},  produces an example which, by adapting the above argument, can be proved to have the following properties.  The 2-cells of its presentation 2-complex embed in its universal cover.  It  \emph{almost}\footnote{We thank the referee for  pointing out that the  disk diagram
for $[a_1,b_1]$ has area 2 but contains no spurs, shells, or cutcells.
The argument of Theorem 1 shows this counterexample is unique.} satisfies  the generalized Dehn property in the sense of definition (1) or, equivalently, (2) of cutcell.   It admits arbitrarily large minimal area diagrams which fail  the generalized Dehn property with cutcells   as per definition (3).

 For example, in the diagram for $[a_1^2,b_1^2]$ of Figure~\ref{third example} the triangle labelled $R$ is a cutcell in the sense of definition (1) or (2), but this diagram contains no spurs or shells and it has no   cutcells in the sense of (3).  
\end{Rem}

\begin{figure}[h] \centering
\begin{tikzpicture}[scale=0.6]
  \node (A) at (0,0) {$\bullet$};
	\node (C) at (3,0) {$\bullet$};
	\node (F) at (0,3) {$\bullet$};
	\node (H) at (3,3) {$\bullet$};
	\node (AA) at (3,0) {$\bullet$};
	\node (CC) at (6,0) {$\bullet$};
	\node (FF) at (3,3) {$\bullet$};
	\node (HH) at (6,3) {$\bullet$};
	\node (TA) at (0,3) {$\bullet$};
	\node (TC) at (3,3) {$\bullet$};
	\node (TF) at (0,6) {$\bullet$};
	\node (TH) at (3,6) {$\bullet$};
	\node (TAA) at (3,3) {$\bullet$};
	\node (TCC) at (6,3) {$\bullet$};
	\node (TFF) at (3,6) {$\bullet$};
	\node (THH) at (6,6) {$\bullet$};

  \draw[thick,->] (A) -- (C) node[below,midway] {$a_1$};
	\draw[thick,->] (F) -- (H) node[above,midway] {$a_1$};
	\draw[thick,->] (A) -- (F) node[left,midway] {$b_1$};
	\draw[thick,->] (C) -- (H) node[right,midway] {$b_1$};
	\draw[thick,->] (AA) -- (CC) node[below,midway] {$a_1$};
	\draw[thick,->] (FF) -- (HH);
	\draw[thick,->] (AA) -- (FF);
	\draw[thick,->] (CC) -- (HH) node[right,midway] {$b_1$};

	\draw[thick,->] (A) -- (C) node[below,midway] {$a_1$};
	\draw[thick,->] (TF) -- (TH) node[above,midway] {$a_1$};
	\draw[thick,->] (TA) -- (TF) node[left,midway] {$b_1$};
	\draw[thick,->] (TC) -- (TH) node[right,midway] {$b_1$};
	
	\draw[thick,->] (TAA) -- (TCC) node[above,midway] {$a_1$};
	\draw[thick,->] (TFF) -- (THH) node[above,midway] {$a_1$};
	\draw[thick,->] (TAA) -- (TFF);
	\draw[thick,->] (TCC) -- (THH) node[right,midway] {$b_1$};

  \node (12) at (1,2) {$\bullet$};
	\node (21) at (2,1) {$\bullet$};
	\node (42) at (4,2) {$\bullet$};
	\node (51) at (5,1) {$\bullet$};
	\node (15) at (1,5) {$\bullet$};
	\node (24) at (2,4) {$\bullet$};
	\node (45) at (4,5) {$\bullet$};
	\node (54) at (5,4) {$\bullet$};
	\draw[thick,->] (TA) -- (12) node[above right,midway] {}; %{$c_1$};
	\draw[thick,->] (12) -- (21) node[above right,midway] {}; %{$c_2$};
	\draw[thick,->] (21) -- (AA) node[above right,midway] {}; %{$c_3$};
	\draw[thick,->] (TC) -- (42) node[above right,midway] {}; %{$c_1$};
	\draw[thick,->] (42) -- (51) node[above right,midway] {}; %{$c_2$};
	\draw[thick,->] (51) -- (CC) node[above right,midway] {}; %{$c_3$};
	\draw[thick,->] (TF) -- (15) node[above right,midway] {}; %{$c_1$};
	\draw[thick,->] (15) -- (24) node[above right,midway] {}; %{$c_2$};
	\draw[thick,->] (24) -- (TAA) node[above right,midway] {}; %{$c_3$};
	\draw[thick,->] (TH) -- (45) node[above right,midway] {}; %{$c_1$};
	\draw[thick,->] (45) -- (54) node[above right,midway] {}; %{$c_2$};
	\draw[thick,->] (54) -- (TCC) node[above right,midway] {}; %{$c_3$};
	\node (K) at (2.1,2.1) {$R$};

\end{tikzpicture} \caption{A diagram over presentation \eqref{third presentation}.} \label{third example}
\end{figure}
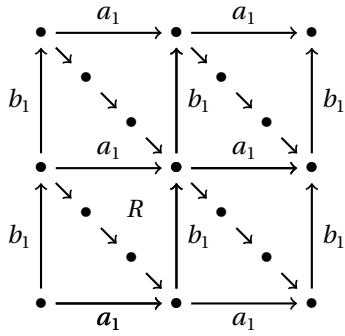

\begin{Rem}
Wise \cite{wisep} proposes imposing an additional hypothesis on $X$.
Say there are \emph{no big pieces} in $X$ if any common boundary arc between two 2-cells is always strictly shorter than half the shorter perimeter of the two 2-cells.  

This \emph{no big pieces} condition in conjunction with the  generalized Dehn property does not imply  a linear isoperimetric inequality.  Our argument  for Theorem \ref{main_counterexample} establishes an example in the form of the presentation 2-complex (shown in Figure~\ref{figure four}) of
\begin{equation}
\langle \, a_1, a_2, b_1, b_2, c_1, c_2, c_3, d_1, d_2 \ | \  a_2b_1b_2a_2^{-1}a_1^{-1}b_2^{-1}c_1c_2c_3,\ b_1c_1d_1^{-1},\ d_1c_2d_2^{-1}, \,d_2c_3a_1^{-1} \, \rangle \ \cong \ \mathbb{Z}^2\ast F_4. \label{pres4}
\end{equation}
\end{Rem}

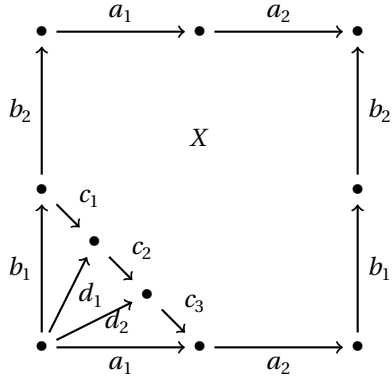
\begin{figure}[h] \centering
\begin{tikzpicture}[scale=0.7]
  \node (A) at (0,0) {$\bullet$};
	\node (B) at (3,0) {$\bullet$};
	\node (C) at (6,0) {$\bullet$};
	\node (D) at (0,3) {$\bullet$};
	\node (E) at (6,3) {$\bullet$};
	\node (F) at (0,6) {$\bullet$};
	\node (G) at (3,6) {$\bullet$};
	\node (H) at (6,6) {$\bullet$};
	\node (I) at (1,2) {$\bullet$};
	\node (J) at (2,1) {$\bullet$};
	\node (K) at (3,4) {$X$};
  \draw[thick,->] (A) -- (B) node[below,midway] {$a_1$};
	\draw[thick,->] (F) -- (G) node[above,midway] {$a_1$};
	\draw[thick,->] (B) -- (C) node[below,midway] {$a_2$};
	\draw[thick,->] (G) -- (H) node[above,midway] {$a_2$};
	\draw[thick,->] (A) -- (D) node[left,midway] {$b_1$};
	\draw[thick,->] (C) -- (E) node[right,midway] {$b_1$};
  \draw[thick,->] (D) -- (F) node[left,midway] {$b_2$};
	\draw[thick,->] (E) -- (H) node[right,midway] {$b_2$};
	\draw[thick,->] (D) -- (I) node[above right,midway] {$c_1$};
	\draw[thick,->] (I) -- (J) node[above right,midway] {$c_2$};
	\draw[thick,->] (J) -- (B) node[above right,midway] {$c_3$};
	\draw[thick,->] (A) -- (I) node[right,midway] {$d_1$};
	\draw[thick,->] (A) -- (J) node[right,midway] {$d_2$};
\end{tikzpicture} \caption{The presentation 2-complex of \eqref{pres4}.} \label{figure four} 
\end{figure}

\begin{Rem} 
The question remains whether  some reasonable notion of \emph{cutcell}   can be added to  the definition of the Dehn property without losing the linear isopermetric inequality.  We have no local condition to suggest.  The best we can see stems from:   
 \begin{proposition}
Suppose a finite connected 2-complex $X$ satisfies the property that every  minimal disk diagram $D\to X$  that is non-trivial (i.e.\ $D$ is not a single vertex)    has a vertex, edge, or face  $R$ for which the closure of  $D-R$ is a union of disk diagrams each with boundary circuit strictly shorter than the boundary circuit  of $D$.
Then $\widetilde{X}$ enjoys a linear isoperimetric inequality.
\end{proposition}
 
 In the following, let $\Area(w)$ denote the minimal area among disk
diagrams with boundary circuit $w$.
The \emph{Dehn function} $\varphi(n)$ is the maximal $\Area(w)$ among
all circuits $w$ of length at most $n$.

\begin{proof} 
Let $c$ be the length of the longest boundary circuit of a 2-cell in $X$.  
Suppose $w$ is a circuit  of length  $n  > 0$.  Let $D$ be a minimal area disk diagram with boundary $w$.
We have  $\Area(w) \leq 1+ \sum_i \varphi(n_i)$ where $\sum_i{n_i} \leq n+c$ and each $n_i<n$. 
 
 Define $f:\mathbb{N}\to\mathbb{N}$ recursively by $f(0)=0$, $f(n)=1+\max\sum_{i\geq1}{f(n_i)}$ where the maximum is taken over all finite sequences of non-negative integers
$\{n_i\}$ satisfying $\sum_i{n_i}\leq n+c$ and $n_i<n$.
Then $\varphi(n)\leq f(n)$ by the previous paragraph, so it suffices to demonstrate a linear upper bound on $f$.
We show by induction on $n$ that $f(n-1)$ appears as a summand in a maximum sum defining $f(n)$ and that $f(x)-f(x-1)$ is nondecreasing on $[1,n]$.
The base case $n=1$ is trivial: $f(1)=1=1+f(0)$.
For the induction step, note that  if $n-1$ does not occur among the $n_i$, then the nondecreasing condition implies we can increment some $n_i$ by 1 at the cost of decrementing another, until $n-1$ occurs; hence the first conjunct of the induction hypothesis follows from the second.
Now consider a sequence $\{\widehat{n_i}\}$ such that $f(n-1)=1+\sum_{i\geq1}{f(\widehat{n_i})}$.
By the induction hypothesis, we may assume $\widehat{n_1}=n-2$.
Setting $n_1=n-1$ and $n_i=\widehat{n_i}$ for $i\neq1$, we find $$f(n) \ \geq \ 1+\sum_{i\geq1}{f(n_i)} \ = \ 1+f(n-1)+\sum_{i\geq2}^n{f(\widehat{n_i})}
 \ = \  f(n-1)-f(n-2)+1+\sum_{i\geq 1}{f(\widehat{n_i})} \ = \ 2f(n-1)-f(n-2)$$ whence $f(n)-f(n-1)\geq f(n-1)-f(n-2)$.
Therefore $f(x)-f(x-1)$ is nondecreasing on $[1,n]$, completing the induction. 
 
Now that we know there is a sum defining $f(n)=1+ \sum_{i\geq1}{f(n_i)}$ with $n_1 = n-1$, we deduce that there is such a sum in which
the remaining terms satisfy $\sum_{i\geq2}{n_i} = c+1$.
Let $K$ be the maximum of $\sum_{j}{f(m_j)}$ over all $\{m_j\}$ with $\sum_{j}{m_j} = c+1$.
Then   $f(n)=1+ f(n-1) + K$ for all $n \geq c+2$.    
(The bound $n\geq c+2$ means the condition $n_i<n$ automatically holds given that $\sum_{i=2}^c{n_i}=c+1$.)
Taking $n=c+2$ we find $K=f(c+2)-f(c+1)-1$. 
Therefore $f(n)=f(n-1)+f(c+2)-f(c+1)$ when $n\geq c+2$. 
The desired linear upper bound on $f$ follows.  \end{proof}
 
If $R$ is a spur or shell, then it will satisfy the condition of the proposition.  So any notion of cutcell which also satisfies the condition of the proposition will give a generalized Dehn property which implies a linear isoperimetric inequality.
\end{Rem}

\begin{Acknowledgements}  We thank Jonah Gaster and Dani Wise for sharing their conjecture with us and for engaging discussions.   We also thank an anonymous referee for a careful reading.
\end{Acknowledgements}


\begin{thebibliography}{1}
 
\bibitem{gasterwise}
J.~{Gaster} and D.~T. {Wise}.
\newblock {Bicollapsibility and groups with torsion,
  \href{http://front.math.ucdavis.edu/1810.12377}{\texttt{arXiv:1810.12377}}},
\newblock   October 2018.

\bibitem{wisep}
D.~T. {Wise}.
\newblock {Private communication},
\newblock    2018.

\end{thebibliography}
\end{document}